\documentclass{dmgt}

\newauthor{%
Pouria Salehi Nowbandegani}{%
P. Salehi}{%
Department of Mathematics,\\
Shiraz University,\\
Shiraz 71454, Iran}[%
pouria.salehi@gmail.com]

\newauthor{%
Hossein Esfandiari}{%
H. Esfandiari}{%
Department of Computer Science, \\
University of Maryland College Park,\\
College Park, MD 20742,}[%
hossein@cs.umd.edu]

\newauthor{%
Mohammad Hassan Shirdareh Haghighi}{%
M. H. Shirdareh}{%
Department of Mathematics,\\
Shiraz University,\\
Shiraz 71454, Iran}[%
shirdareh@susc.ac.ir]

\newauthor{%
Khodakhast Bibak}{%
Kh. Bibak}{%
Department of Combinatorics and Optimization,\\
University of Waterloo,\\
Waterloo, Ontario, Canada N2L 3G1}[%
kbibak@uwaterloo.ca]

\title{On the Erd\H{o}s-Gy\'{a}rf\'{a}s
Conjecture in Claw-free Graphs}


\classnbr{C5038, C5038}

\begin{document}

\begin{abstract}
The Erd\H{o}s-Gy\'{a}rf\'{a}s conjecture states that every graph
with minimum degree at least three has a cycle whose length is a
power of 2. Since this conjecture has proven to be far from reach,
Hobbs asked if the Erd\H{o}s-Gy\'{a}rf\'{a}s conjecture holds in
claw-free graphs. In this paper, we obtain some results on this
question, in particular for cubic claw-free graphs.
\end{abstract}

\section{Introduction}

All graphs in this paper are assumed to be simple, that is,
without any loops and multiple edges. Let us first recall here
briefly some notation and terminology we will need in this paper.
We denote by $\delta=\delta(G)$ the minimum degree of the the
vertices in the graph $G=(V,E)$. A {\it $uv$-path\/} is a path
having the vertices $u$ and $v$ as its ends. The length of a path
$P$ (or a cycle $C$) is denoted by $l(P)$ (resp. $l(C)$). Also, we
denote the distance between the vertices $u$ and $v$ by $d(u, v)$,
that is the length of a shortest $uv$-path. A graph that does not
contain a particular graph $H$ as an induced subgraph is called
{\it $H$-free\/}. The complete bipartite graph $K_{1,3}$ is
referred to as a {\it claw\/}; so a graph is called {\it
claw-free\/} if it does not have $K_{1,3}$ as an induced subgraph.
A {\it triangle\/} is a cycle of length three. A {\it chord\/} of
a cycle $C$ is an edge between two vertices of $C$ which are not
adjacent in $C$. By a {\it hole\/} we mean a chordless cycle of
length at least four. A hole of length $n$ is called an $n$-hole.

\vspace{1mm}

Several questions on cycles in graphs have been posed by Erd\H{o}s
and his colleagues (see, e.g., \cite{BON}). In particular, in 1995
Erd\H{o}s and Gy\'{a}rf\'{a}s \cite{ERD} asked:

\vspace{1mm}

{\it If $G$ is a graph with minimum degree at least three, does
$G$ have a cycle whose length is a power of \normalfont{2}?\/}

\vspace{1mm}

This is known as the {\it Erd\H{o}s-Gy\'{a}rf\'{a}s conjecture\/}.
In fact, Erd\H{o}s and Gy\'{a}rf\'{a}s \cite{ERD} said that ``we
are convinced now that this is false and no doubt there are graphs
for every $r$ every vertex of which has degree $\geq r$ and which
contain no cycle of length $2^{k}$, but we never found a
counterexample even for $r = 3$".

There seems to be very little published on the
Erd\H{o}s-Gy\'{a}rf\'{a}s conjecture. Markstr\"{o}m \cite{MAR}
(via computer searches) asserted that any cubic counterexample
must have at least 30 vertices. Salehi Nowbandegani and Esfandiari
\cite{SAL} prove that any bipartite counterexample must have at
least 32 vertices.

More generally, Erd\H{o}s asked does there exist an integer
sequence $a_1, a_2, a_3, \cdots $ with zero density, and a
constant $c$ such that every graph with average degree at least
$c$ contains a cycle of length $a_i$ for some $i$. This question
is answered affirmatively by Verstra\"{e}te \cite{VER}.

 Hobbs asked if the Erd\H{o}s-Gy\'{a}rf\'{a}s conjecture holds in
claw-free graphs \cite{DASH}. Shauger \cite{SHA} proved the
conjecture for $K_{1,m}$-free graphs having minimum degree at
least $m+1$ or maximum degree at least $2m-1$. Also, Daniel and
Shauger \cite{DASH} proved it for planar claw-free graphs. In this
paper, we investigate claw-free graphs with $\delta \geq$4 and
cubic claw-free graphs.

\section{Two-power Cycle Lengths in Claw-free Graphs}
Our first theorem concerns claw-free graphs with $\delta \geq 3$.

\begin{theorem}
Suppose that $G$ is a claw-free graph with $\delta \geq 3$. Then
$G$ has a cycle whose length is $2^{k}$, or $3\cdot 2^{k}$, for
some positive integer $k$.
\end{theorem}
To prove Theorem 1 we need the following lemma.
\begin{lemma}
Let $G$ be a graph with $\delta \geq 3$. If $G$ does not have
$C_4$ as a subgraph, then for some $n\geq 5$, it has an $n$-hole.
\end{lemma}
\begin{proof}
It is known that every graph with $\delta \geq 2$ contains a cycle
of length at least $\delta+1$ (see, e.g., \cite[Exercise
2.1.5]{BOMU}). Thus $G$ has a cycle $D_1$ of length $n_1\geq 5$.
If $n=5$, $D_1$ must clearly be chordless. If $n>5$, and $D_1$ has
no chord, we are finished, so suppose $D_1$ has a chord. The chord
separates $D_1$ into two shorter cycles, non of which have length
4, by assumption. Thus at least one of these two cycles, say
$D_2$, must have length 5$\leq n_2<n_1$. Since $G$ is finite, we
must by repeating this argument eventually find a chordless cycle
$D_k$ of length $n_k\geq$5.
\end{proof}

\begin{dnt}
We call an edge  of a graph {\it triangulated} if it is contained
in a triangle. Also if such a triangle is unique, we call the edge
{\it uniquely triangulated}.
\end{dnt}

Now we are ready to prove Theorem 1.

\vspace{2mm}
\begin{proof}[Proof of Theorem 1]
If $G$ has a cycle of length four,  the theorem holds, with $k=$2.
We may therefore assume that $G$ does not contain any $C_4$. Thus,
by Lemma 2, for some $n\geq 5$, $G$ has an $n$-hole. Let $C:\;
a_1a_2\ldots a_sa_1$, $s\geq 5$, be a smallest hole in $G$. Since
$\delta \geq 3$ and $C$ is a hole, each vertex of $C$ has a
neighbour in $G-V(C)$. For $i$, ($1\leq i \leq s$), suppose that
$a_ib_i \in E(G)$, where $a_i \in C$ and $b_i \in V(G)\setminus
V(C)$. Then either $a_{i-1}b_i \in E(G)$, or $a_{i+1}b_i \in
E(G)$, because $G$ is claw-free. Now we show that $b_i \not= b_j$
if $|j-i|\geq 2$. To get a contradiction, fix $i$ and let  $a_j$
be the first vertex of $C$ after $a_i$ such that $b_i = b_j=b$,
$|j-i|\geq 2$. If $j-i=2$, then we get the $C_4:\;
a_ia_{i+1}a_{i+2}ba_i$, which is absurd. If $|j-i|> 2$, then we
get the hole $a_{i+1}\ldots a_jba_{i+1}$ which is certainly
smaller than $C$ (note that we don't reject the case that this
hole may be a $C_4$).

 Therefore, it follows that  every other edge of $C$ is uniquely
 triangulated; we mark them.
  Moreover,  the third
vertices of the corresponding triangles are disjoint.  Note also
that $s$ is even. Consequently,  we find cycles of lengths
$s,s+1,\ldots, \frac{3}{2}s$ by traversing $C$ such that as we
reach a marked edge, we pass it directly or through  the third
vertex of its corresponding triangle. Since either there exists a
$2^k$ or a $3\cdot 2^{k-1}$ between $s$ and $\frac{3}{2}s$, the
proof is complete.
\end{proof}

 As mentioned above, Shauger \cite{SHA} proved the
Erd\H{o}s-Gy\'{a}rf\'{a}s conjecture for $K_{1,m}$-free graphs
having minimum degree at least $m+1$ or maximum degree at least
$2m-1$. Theorem 5 improves on the result of Shauger in claw-free
graphs. First we state the following proposition. We omit the easy
proof.

\vspace{2mm}

\begin{prop} In a {\normalfont 4}-regular claw-free graph
which does not contain $C_4$, every edge is uniquely triangulated.
\end{prop}

\begin{lemma}
Let $G$ be a {\normalfont 4}-regular claw-free graph which does
not contain $C_4$ and $v$ be a vertex of $G$. Let $C$ be a
smallest n-hole in $G$ containing $v$, $n\geq5$. Then for every
edge $xy$ of $C$, the third vertex $z=z(xy)$ of the corresponding
triangle of $xy$  is out of $C$. Furthermore, if $uw\not = xy$ are
two edges of $C$, then $z(uw)\not = z(xy)$.
\end{lemma}

\begin{proof}
First note that since $C$ is a hole, for every edge $xy$ in $C$,
$z=z(xy) \notin C$.  Let $uw$ and $wx$ be two consecutive edges in
$C$. If $z=z(uw)=z(wx)$, then we get the $C_4: \; uwxzu$. Hence
$z(uw)\not =z(wx)$. Suppose that $uw$ and $xy$ are two
non-consecutive edges in $C$ and suppose $C$ traverses the
vertices in order $u,w,x,y$, and then $v$. Let $Q$ be the $yvu$
segment of $C$. Now if $z=z(uw)=z(xy)$, then the cycle $uQyzu$ is
a smaller hole containing $v$; unless $u$ and $y$ are adjacent in
$C$ (and hence $v$ is one of them). But in this case, we see that
$uzxyu$ is a $C_4$ in $G$. This contradiction shows that
$z(uw)=z(xy)$ for $uw \not = xy$ is impossible.
\end{proof}

\begin{theorem}
Let $G$ be a claw-free graph with $\delta \geq 4$, which does not
contain $C_4$. Then every non-cut vertex of $G$ lies on a cycle
whose length is a power of \normalfont{2}.
\end{theorem}

\begin{proof}
Since $\delta \geq 4$ and $G$ is claw-free, if $G$ has a vertex
with degree at least 5, then this vertex lies on a $C_4$; so we
can assume that $G$ is 4-regular. Suppose that $v$ is a non-cut
vertex of $G$ and let $w$, $x$, $y$, and $u$ be its neighbours.
Hence, $G-v$ is connected. In view of $G$ is claw-free,  we can
assume that $wu, xy\in E(G)$. Let $P_1$, $P_2$, $P_3$, and $P_4$
be the shortest $wy$-path, $wx$-path, $xu$-path, and $yu$-path in
$G-v$, respectively. Also, without loss of generality assume that
$l(P_1)=\min \{l(P_1), l(P_2), l(P_3), l(P_4)\}$. The path $P_1$
together with the edges $vw$ and $vy$ make a cycle $C$. Clearly,
$l(P_1)>1$, otherwise $ywuvy$ will be a $C_4$. Therefore,
$l(C)=s\geq 5$. Since $P_1$ was the shortest path among $P_1$,
$P_2$, $P_3$, and $P_4$,  we see that neither $x$ nor $u$ are in
$P_1$ and, in fact, $C$ is the shortest non-triangle hole
containing the vertex $v$; for if $v$ lies on another non-triangle
shorter hole, then two of its neighbours would have distance less
than $l(P_1)$ in $G-v$.  By Lemma 4, each edge of $C$ is uniquely
triangulated  such that the
 third vertex of its corresponding triangle is not on $C$ and
 this correspondence is one to
one. Since $l(C)=s$, then $G$ contains cycles of lengths
$s,s+1,\ldots, 2s$. For, as in the proof of theorem 1, when we
traverse  the vertices of $C$, we can either pass the two ends of
every edge directly or through the third vertex of its
corresponding triangle.

This implies that $G$ has a cycle containing $v$ whose length is
$2^{k}$, for some $k\geq 3$.

\end{proof}

\section{The Erd\H{o}s-Gy\'{a}rf\'{a}s Conjecture in Cubic Claw-free Graphs}

In this section, we investigate the Erd\H{o}s-Gy\'{a}rf\'{a}s
conjecture in cubic claw-free graphs. Indeed, we discuss on the
cubic claw-free graphs for which the Erd\H{o}s-Gy\'{a}rf\'{a}s
conjecture possibly does not hold.

Suppose that $G$ is a cubic claw-free graph that does not contain
$C_4$. Let $v$ be an arbitrary vertex of $G$, and let its
neighbours be $x$, $y$, and $z$. Since $G$ is claw-free, so we can
assume that $xy\in E(G)$. Thus, $xz,yz\notin E(G)$; otherwise a
$C_4$ appears. Let $x_1$ and $y_1$ be respectively the other
neighbours of $x$ and $y$. Easily we see that $x_1\not=y_1$.
Therefore, for every vertex there exists a unique triangle
containing it, such that the other neighbours of its vertices are
distinct. Hence $G$ consists of some vertex-disjoint triangles
which are connected by a perfect matching of $G$. Furthermore, if
two vertices from two triangles are matched, then there is no more
link between these two triangles, again because we have no $C_4$
in $G$. This means if we look locally at the graph, we see a
triangle together with three appended edges, such that these edges
connect to three disjoint triangles. Now define $\hat(G)$ to be
the graph whose vertices are triangles of $G$ and two vertices are
adjacent in $\hat{G}$ whenever their corresponding triangles in
$G$ are linked by an edge. The graph $\hat{G}$ is then a simple
cubic graph. We can imagine $\hat{G}$ as a graph obtained from $G$
by shrinking each triangle to a vertex.

 Conversely, we can start from a simple cubic graph $\hat{G}$ and replacing
each vertex $v$ with a triangle $T$; linking the three vertices of
$T$ to the three triangles corresponding to the three neighbours
of $v$. This procedure results in a cubic claw-free graph $G$
without $C_4$. To sum up, we have the following proposition.

\begin{prop} The mapping $G\leftrightarrow \hat{G}$  is a one to one correspondence
between simple cubic graphs and simple cubic claw-free graphs
without $C_4$.
\end{prop}

\begin{cor}
If $\hat{G}$ contains a cycle of length $k$, then  this cycle
provides cycles of lengths $2k, 2k+1, \ldots, 3k$ in $G$.
\end{cor}
\begin{proof}
Consider a cycle $\hat{C}$ of length $k$ in $\hat{G}$. The
 subgraph $S$ of $G$ corresponding to $\hat{C}$ consists of a cycle of
 length $2k$ such that every other edge of it is triangulated. Hence
 we can find cycles of lengths $2k, 2k+1, \ldots, 3k$ in $S$.
\end{proof}

Based on proposition 6 and Corollary 7,  we think the following
conjecture is true.

\begin{con} Every cubic graph contains a cycle of length $l$ such that
$2l\leq 2^{k}<3l$, for some positive integer $k$.
\end{con}

\vspace{2mm} If this conjecture holds, it will lead to a proof of
the Erd\H{o}s-Gy\'{a}rf\'{a}s conjecture in cubic claw-free
graphs. Also note that this conjecture can be easily deduced from
the Erd\H{o}s-Gy\'{a}rf\'{a}s conjecture. But for simplicity, we
restrict ourselves to cubic graphs, and the length  of the desired
cycle has a very wide range.

At the end, we investigate minimal cubic claw-free graphs which
possibly have no cycle with length a power of 2.

\begin{theorem}
Any counterexample to the Erd\H{o}s-Gy\'{a}rf\'{a}s conjecture in
cubic claw-free graphs must have at least \normalfont{114}
vertices.
\end{theorem}

\begin{proof}
Let $G$ be a claw-free cubic graph of order $3n$. Then ${\hat{G}}$
(defined in proposition 6) is a cubic graph of order $n$. By
corollary 7, if ${\hat{G}}$ contains a cycle of length $l$, where
$l\in \{2,3,4,6,7,8\}$, then the Erd\H{o}s-Gy\'{a}rf\'{a}s
conjecture holds for $G$. So let us assume that ${\hat{G}}$ does
not contain such cycles. Let $v_0$ be a vertex of ${\hat{G}}$. We
consider $\{v_0\}$ as level 0, and define level $i$, $i\geq 1$, as
the set
$$L_i=\{v\in V({\hat{G}}) \; :\;\;d(v,v_0)=i\}.$$
Clearly, $L_1$ is an independent set. It is easy to see that the
subgraph induced by $L_2$ has at most one edge. One can check that
if the subgraph induced by $L_2$ has no edge, then the subgraph
induced by $L_3$ has at most three edges, and if the subgraph
induced by $L_2$ has one edge, then the subgraph induced by $L_3$
has at most one edge. No two elements of $L_3$ have common
neighbours in $L_4$, because otherwise, ${\hat{G}}$ contains the
cycles of lengths 2, 4, 6, or 8. An easy calculation shows that
${\hat{G}}$ has at least 38 vertices. Consequently, any
counterexample for the Erd\H{o}s-Gy\'{a}rf\'{a}s conjecture must
have at least $3\times 38=114$ vertices.
\end{proof}

\begin{center}{\bf{Acknowledgments}}
\end{center}
The authors would like to thank anonymous referees for helpful
mathematical and grammatical comments.

\end{document}